\documentclass[A4paper, 12pt]{amsart}

\usepackage{amsmath, amssymb}
\usepackage{amsthm}
\usepackage{multicol}

\newtheorem{thm}[equation]{Theorem}
\newtheorem{prop}[equation]{Proposition}
\newtheorem{lem}[equation]{Lemma}
\newtheorem{sublem}[equation]{Sublemma}

\newtheorem{mthm}{Theorem}
\newtheorem{mcor}[mthm]{Corollary}
\newtheorem{step}{Step}
\newtheorem{clm}{Claim}

\theoremstyle{definition}
\newtheorem{defi}[equation]{Definition}
\newtheorem{ex}[equation]{Example}

\theoremstyle{remark}
\newtheorem{rem}[equation]{Remark}

\newcommand{\arrowprod}[2]{\left\langle \overrightarrow{#1}, \overrightarrow{#2} \right\rangle_\kappa}
\newcommand{\arrowprodnone}[2]{\left\langle \overrightarrow{#1}, \overrightarrow{#2} \right\rangle}
\newcommand{\ip}[2]{\left\langle {#1}, {#2} \right\rangle}
\newcommand{\norm}[1]{\left| #1 \right|}
\newcommand{\bra}[1]{\left\{ #1 \right\}}
\newcommand{\cosq}[2]{{\rm cosq} (\overrightarrow{#1}, \overrightarrow{#2})}
\newcommand{\cangle}{\tilde{\angle}_\kappa}
\newcommand{\zengle}{\tilde{\angle}_0}
\newcommand{\onegle}{\tilde{\angle}_1}
\newcommand{\arrow}[2]{{\uparrow^{#1}_{#2}}}
\newcommand{\barrow}[2]{\bra{\uparrow^{#1}_{#2}}}

\renewcommand{\H}{\mathcal{H}}
\newcommand{\R}{\mathbb{R}}
\renewcommand{\S}{\mathbb{S}}
\renewcommand{\[}{\left[}
\renewcommand{\]}{\right]}
\renewcommand{\(}{\left(}
\renewcommand{\)}{\right)}
\newcommand{\supp}{{\rm Supp}\,}
\newcommand{\conv}{\overline{{\rm conv}}\,}
\newcommand{\pack}{{\rm pack}_q}

\renewcommand{\k}{{\kappa}}
\renewcommand{\phi}{\varphi}
\renewcommand{\epsilon}{\varepsilon}
\newcommand{\tx}{\tilde{x}}
\newcommand{\ty}{\tilde{y}}
\newcommand{\tz}{\tilde{z}}
\newcommand{\tw}{\tilde{w}}
\newcommand{\tp}{{\tilde{p}}}
\newcommand{\txi}{\tilde{\xi}}
\newcommand{\teta}{\tilde{\eta}}

\newcommand{\cx}{{\check{x}}}
\newcommand{\cxi}{{\check{\xi}}}
\newcommand{\clambda}{\check{\lambda}}
\newcommand{\mua}{{\mu^\alpha}}
\newcommand{\xia}{{\xi^\alpha}}
\newcommand{\cxia}{{\check{\xi}^\alpha}}
\newcommand{\etaa}{{\eta^\alpha}}

\title[A rigidity theorem in Alexandrov spaces]{A rigidity theorem in Alexandrov spaces with lower curvature bound}
\author{Takumi Yokota}
\address{Graduate School of Pure and Applied Sciences\\
University of Tsukuba\\
305-8571 Tsukuba\\
Japan}
\email{takumiy@math.tsukuba.ac.jp}
\keywords{Alexandrov space, Triangle comparison, Quadruple comparison, infinite dimension}
\subjclass[2000]{53C23 (primary), 53C24, 54E50 (secondary)}
\thanks{The author was partially supported by Research Fellowships of the Japan Society for the Promotion of Science for Young Scientists (19$\cdot$9377).}

\begin{document}
\maketitle

\begin{abstract}
Distance functions of metric spaces with lower curvature bound,
by definition, enjoy various metric inequalities; triangle comparison, quadruple
comparison and the inequality of Lang--Schroeder--Sturm. The purpose of this
paper is to study the extremal cases of these inequalities and to prove rigidity
results. The spaces which we shall deal with here are Alexandrov spaces which
possibly have infinite dimension and are not supposed to be locally compact.
\end{abstract}

\section{Introduction}
The main object of the present paper is Alexandrov spaces with curvature bounded
below by a real number $\k \in \R$.
The definition is given in Definition~\ref{def:Alex} below. These
metric spaces are known to have the following property:
\begin{quotation}
All triangles in the space are thicker than the ones in the model surface $M^2_\k$.
\end{quotation}
Now let us make this statement more precise.

Throughout this paper, $M_\k$ stands for the infinite dimensional model space,
that is, $M_\k$ is the Elliptic, Euclidean, or Hyperbolic cone (e.g.~\cite{BBI}), depending
on the sign of $\k\in\R$, over the unit sphere of some infinite dimensional separable
Hilbert space. The model surface, i.e., the simply connected complete surface of
constant curvature $\k$, is denoted by $M^2_\k$.

We will use $[A,B]$ to denote the subset of a metric space $(X, d)$ defined by
$$[A,B] := \{x \in X \bigm| d(a, b) = d(a, x) + d(x, b) \text{ for some } a \in A, b \in B \}$$
for two subsets $A, B \subset X$. We also use $[a, B] := [\{a\}, B]$, $[a, b] := [\{a\}, \{b\}]$ and
$(a, b) := [a, b] \setminus \{a, b\}$ for any $a, b \in X$.

We return to the above situation. Let $(X, d)$ be an Alexandrov space with curvature
$\ge\k$ and $\Delta xyz := \{x, y, z\}$ be any three point set in $X$. When $\k > 0$ is positive,
we impose that the perimeter ${\rm peri}(x, y, z) := d(x, y) + d(y, z) + d(z, x)$ of $\Delta xyz$
is less than $2\pi/\sqrt\k$. Then we can find an isometric embedding of $\Delta xyz$ into
the model surface $M^2_\k$ and let  $\tilde\Delta xyz := \{\tx, \ty, \tz\}$ be the image of $\Delta xyz$ in $M^2_\k$.
We always use the tilde to indicate the points of the model space $M_\k$ (or $M^2_\k$)
corresponding to the ones of $X$.

Suppose that we have a point $w$ of $X$ with $w \in (y, z)$ and let $\tw \in (\ty, \tz)$ be the
corresponding point of $M^2_\k$. Then it follows from the very definition of the lower
curvature bound and Alexandrov's lemma (e.g.~\cite{BBI}) that
\begin{equation}\label{ineq:moti}
d(x, w) \ge d(\tx, \tw).
\end{equation}
The motivation of this work comes from the concern about the case where the
equality holds in~\eqref{ineq:moti}, and we shall establish a general rigidity theorem, which is applicable to various
situations including the equality case of~\eqref{ineq:moti}.

There have been intense studies of the geometry of Alexandrov spaces by many
authors. This is partly because they appear naturally as Gromov--Hausdorff limits
of sequences of Riemannian manifolds with uniform lower sectional curvature
bound. For instance, the theory of Alexandrov spaces, as well as Hamilton's
Ricci flow, was extensively utilized in Perelman's proof of the Poincar`e conjecture
(e.g.~\cite{MT}).

However, in this paper, we are concerned with the geometry of general,
namely, not necessarily of finite dimension nor locally compact, 
Alexandrov spaces itself and are going to prove a rigidity theorem in this setting.
We also generalize a result which is known in the finite dimensional case to 
the one which is also available in the infinite dimensional case;
see Corollaries~\ref{cor:D} and \ref{cor:E}.
As a precursor to this work, Mitsuishi~\cite{Mi} found out that
the splitting theorem for non-negatively curved Alexandrov spaces is also valid
in our setting.

The organization of this paper is as follows: In the next two sections, we recall
the definitions and properties of Alexandrov spaces. Then the main theorems of
this paper are formulated in Section~\ref{sect:4}. After establishing preliminary results in
Section~\ref{sect:5}, we describe the proofs of our main theorems in Sections~\ref{sect:6}--\ref{sect:8}. The final
section is devoted to a few applications of our results.

We refer the reader to \cite{BGP} and \cite{BBI, Pl:Surv, Sh} for basics of the theory of
Alexandrov spaces with curvature bounded below.

\section{Definitions}\label{sect:2}
We first define a notion that its curvature is bounded below for arbitrary metric
spaces. In order to do this, we need to set the comparison angle $\cangle(x; y, z) \in [0, \pi]$
for $\k \in \R$ by the law of cosines;
\begin{align*}
\cos \cangle(x; y, z) :=
\begin{cases}
\dfrac{d(x, y)^2 + d(x, z)^2 - d(y, z)^2}{2d(x, y)d(x, z)} &\text{ if } \k = 0;\\[12pt]
\dfrac{C_\k(d(y, z)) - C_\k(d(x, y))C_\k(d(x, z))}{\k S_\k(d(x, y))S_\k(d(x, z))} &\text{ if } \k \ne 0
\end{cases}
\end{align*}
for three points $x, y, z$ of a metric space $(X, d)$, with ${\rm peri}(x, y, z) < 2\pi/\sqrt\k$ if
$\k > 0$. Here
$$S_\k(r) :=
\begin{cases}
\sin(\sqrt\k r) / \sqrt\k &\text{ if } \k > 0;\\
\sinh(\sqrt{-\k}r) / \sqrt{-\k} &\text{ if } \k < 0
\end{cases}$$
and $C_\k(r) := S^\prime_\k(r)$.

\begin{defi}\label{def:curv}
Let $\k \in \R$ be a real number. We say that a metric space $(X, d)$
has curvature bounded below by $\k$, (or shortly $\ge \k$), if the following holds:
\begin{equation}\label{ineq:quad}
\cangle(x; y, z) + \cangle(x; z,w) + \cangle(x;w, y) \le 2\pi
\end{equation}
for any quadruple $(x; y, z, w)$ consisting of distinct four points of $X$, with size,
i.e., the maximum of perimeters of all triangles, is less than $2\pi/\sqrt\k$ if $\k > 0$. The
condition~\eqref{ineq:quad} is sometimes called the \textit{quadruple condition}.
\end{defi}

\begin{defi}\label{def:Alex}
Let $(X, d)$ be a metric space. We say that $(X, d)$ is an \textit{Alexandrov
space with curvature bounded below by $\k$} if it satisfies the following axioms:
\begin{enumerate}
\item
$(X, d)$ is a complete length space (which is not necessarily a geodesic
space), and
\item
$(X, d)$ has curvature bounded below by $\k$, in the sense of Definition~\ref{def:curv}.
\end{enumerate}

A metric space $(X, d)$ is called an \textit{inner space} if for any points $x, y \in X$,
$t \in (0, 1)$ and $\epsilon > 0$, we can find a point $z \in X$ such that
$$\frac{d(x, z)^2}{t} + \frac{d(z, y)^2}{1 - t} \le d(x, y)^2 + \epsilon.$$

A complete metric space is an inner space if and only if it is a length space. A
\textit{length space} is a metric space in which the distance between any two points is
equal to the infimum of the length of curves connecting these points. A \textit{geodesic
space} is a length space where the infimum is always attained for any two points.
We simply say that $X$ is an Alexandrov space when we do not need to refer to
the lower curvature bound $\k$.
\end{defi}

Some authors assume that the Alexandrov spaces they are dealing with are
finite dimensional or locally compact or geodesic spaces. In this paper, we adopt
the definition that seems to be most general. Here we recall the following implications
(\cite{BGP}): For an Alexandrov space $X$ with lower curvature bound,
\begin{align*}
X \text{ is of finite dimension}
&\implies X \text{ is locally compact}\\
&\implies X \text{ is a geodesic space}.
\end{align*}

It is also known that for any Alexandrov spaces, the Hausdorff dimension and
the covering (or topological) dimension agree (\cite{BGP, PP}).

We pause here to give examples.
The so-called metric transforms offer a large number of examples of metric spaces with lower curvature bound
(cf.~\cite{LV, VW}).
\begin{ex}
Given an arbitrary metric space $(X, d)$ and $\k\in\R$,
we equip $X$ with another metric $d_\k$ defined by $d_\k(x, y) := \phi_\k(d(x, y))$ for $x, y \in X$.
Here, $\phi_\k$ is the function given for fixed $\alpha \in [0, 1/2]$ by
\begin{align*}
\phi_\k(t) :=
\begin{cases}
\cos^{-1}(1- \frac{1}{2}\min\{\sqrt\k\, t^{2\alpha}, 1\}) /\sqrt\k &\text{ if } \k>0;\\ 
t^{\alpha} &\text{ if } \k=0;\\
\cosh^{-1}(1+ \frac{1}{2} \sqrt{-\k}\, t^{2\alpha}) /\sqrt{-\k} &\text{ if } \k<0.
\end{cases}
\end{align*}
Then the resulting metric space $(X, d_\k)$ has curvature $\ge\k$ for any $\alpha \in [0, 1/2]$.
\end{ex} 

The next results give two examples of Alexandrov spaces of non-negative curvature.
\begin{ex}
Let $(X, d)$ be a (locally compact) length space. Then
\begin{enumerate}
\item
the $L^2$-map space $(L^2(M, X; f_0), d_{L^2})$ (Mitsuishi~\cite{Mi}) and
\item
the $L^2$-Wasserstein space $(P_2(X), W_2)$ (e.g. Villani~\cite{Vi})
are Alexandrov spaces of non-negative curvature if and only if so is $(X, d)$.
\end{enumerate}

We refer the reader to the references quoted above for the definitions and the
precise statements. The corresponding result of Part (1) is well-known for CAT$(0)$
spaces.
\end{ex}

We use the term \textit{geodesic} to mean a constant speed (usually unit speed) curve
$\gamma : [0, l] \to X$ whose length is equal to the distance between its end points. We
shall implicitly identify a map $\gamma$ and its image $\gamma([0, l])$.

For any two geodesics $\gamma$ and $\eta : [0, l] \to X$ emanating from the same point
$p = \gamma(0) = \eta(0)$ of a metric space $(X, d)$ with curvature $\ge\k$, we define the \textit{angle}
between $\gamma$ and $\eta$ by
\begin{equation}
\angle(\gamma, \eta) := \lim_{s, t \to 0+} \cangle(p; \gamma(s), \eta(t)).
\end{equation}
The angel is always well-defined and not less than the comparison angle in a
metric space with curvature $\ge\k$, where the monotonicity of the comparison
angle holds: For any $x, y, z, z^\prime \in X$,
$$z^\prime \in (x, z) \implies  \cangle(x; y, z^\prime) \ge \cangle(x; y, z).$$
We need to invoke Alexandrov's lemma (e.g.~\cite{BBI}) to prove this.

Next, we define a space of directions and a tangent cone.
\begin{defi}
Let $X$ be an Alexandrov space with curvature bounded below and
$p$ be a point of $X$. We put $\Sigma^\prime_p$ as the set of all unit speed geodesics emanating
from $p$ with two geodesics identified if the angle between then is $0$. Then the
metric completion $(\Sigma_p, \angle)$ of $(\Sigma^\prime_p, \angle)$ is called the \textit{space of directions} at $p$.

The \textit{tangent cone} $(C_p, |\cdot|)$ at $p$ is, by definition, the Euclidean cone over the
space of directions $(\Sigma_p, \angle)$. We define, on $C_p$, the inner product $\ip{\cdot}{\cdot}$ by
$$\ip{\alpha \cdot \xi}{\beta \cdot \eta} := \alpha\beta \cos\angle(\xi, \eta),$$
and the metric $|\cdot|$ by
$$\norm{\alpha \cdot \xi - \beta \cdot \eta}^2 := \alpha^2 + \beta^2 - 2 \ip{\alpha \cdot \xi}{\beta \cdot \eta},$$
respectively, for any $\alpha \cdot \xi = (\xi, \alpha)$ and $\beta \cdot \eta = (\eta, \beta)$
in $C_p := \Sigma_{p} \times [0, \infty)/ \Sigma_{p} \times \{0\}$.

Following Petrunin~\cite{Pe}, we use $\arrow{q}{p} \in \Sigma_p$ and $\log_{p}q \in C_p$, respectively,
to denote the equivalence class of a
unit speed geodesic from $p$ to $q \in X$ and $d(p, q)\cdot \arrow{q}{p}$.
\end{defi}

Recall that we are dealing with Alexandrov spaces that are merely length
spaces, and there is no way to consider a space of directions in spaces with no
geodesics. The following proposition tells us that our Alexandrov spaces posses
many geodesics (cf. Otsu--Shioya~\cite[Lemma 2.2]{OS}).

\begin{prop}[Plaut~\cite{Pl, Pl:Surv}]
For any point $p$ of an Alexandrov space $X$ with
curvature $\ge\k$, let
$$J_p := \bra{q \in X \bigm| \text{there exists a unique geodesic connecting } p \text{ and } q}.$$
Then $J_p$ contains the following dense $G_\delta$-set:
$$\bigcap_{k\in \mathbb{N}} \bigcup_{x\in X} \bra{y \in X \Bigm| \cangle(y; p, x) > \pi - k^{-1}}.$$
\end{prop}

\section{Properties of Alexandrov spaces}\label{sect:3}
In this section, we collect basic properties of Alexandrov spaces with lower
curvature bound. We will use $I$ to denote an index set consisting of $N := \# I$
elements.

Let us start with the following elementary observation. For any Hilbert space
$(\H, \ip{\cdot}{\cdot})$ and finite sequence $\bra{x_i}_{i \in I} \subset \H$ of points,
the symmetric matrix $(\ip{x_i}{x_j})_{ij}$
is positive semi-definite. Indeed, we have
$$\sum_{i, j\in I} \lambda_i\lambda_j \ip{x_i}{x_j}
= \left\| \sum_{i \in I} \lambda_ix_i \right\|^2
\ge 0$$
for any sequence $\bra{\lambda_i}_{i \in I} \subset \R$ of real numbers. This observation leads us to the
following definition and theorem.

For three points $x, y, z$ of a metric space $(X, d)$ and $\k \in \R$, we put
\begin{equation}
\arrowprod{xy}{xz} := d(x, y)d(x, z) \cos \cangle(x; y, z).
\end{equation}
We define $\arrowprod{xy}{xz} := 0$ if $d(x, y) \cdot d(x, z) = 0$.
When $\k > 0$, $d(x, y) \cdot d(x, z) > 0$
and $\cangle(x; y, z)$ is not well-defined, we let $\arrowprod{xy}{xz} := +\infty$, by definition. We
refer to $\arrowprod{xy}{xz}$ as the \textit{inner product} of $(X, d)$.

\begin{thm}[Sturm~\cite{St}]\label{thm:St}
Let $(X, d)$ be a length space.
Then the following are equivalent.
\begin{enumerate}
\item{\normalfont (Triangle comparison)}
For any three points $x, y, z \in X$, with ${\rm peri}(x, y, z) <
2\pi/\sqrt\k$ if $\k > 0$, a sequence
$\bra{w^k} \in (y, z)$ and the corresponding points
$\tx, \ty, \tz, \tw$ of $M^2_\k$,
\begin{equation}\label{ineq:tri}
d(x, \bra{w^k}) \ge d (\tx, \tw).
\end{equation}

\item{\normalfont (Quadruple comparison)}
For any quadruple $(x; y, z, w) \subset X$, with size
$< 2\pi/\sqrt\k$ if $\k > 0$,
$$\cangle(x; y, z) + (x; z, w) + \cangle(x; w, y) \le 2\pi.$$

\item{\normalfont (Lang--Schroeder--Sturm inequality)}
For any $p \in X$ and finite sequences
$\bra{x_i}_{i \in I} \subset X$ of points and $\bra{\lambda_i}_{i \in I} \subset \R_+$ of positive real numbers,
\begin{equation}\label{ineq:LSS}
\sum_{i, j\in I} \lambda_i\lambda_j \arrowprod{px_i}{px_j} \ge 0.
\end{equation}
\end{enumerate}
\end{thm}

In handling a length space $(X, d)$, we find it convenient to pretend that a
sequence, say $\bra{x^k}\subset X$, of points is a single point; this is essentially equivalent
to work in the ultra limit $\lim_\omega X_i$ of constant sequence $X_i := X$ (e.g.~\cite{Mi}). The
distance is, for example, defined by
$$d(\bra{x^k}, \bra{y^l}) := \lim_{k, l\to\infty} d(x^k, y^l)$$
whenever the limit in the RHS exists. We also let them show up in inequalities.
Inequality~\eqref{ineq:tri} means that
$$\liminf_{k\to\infty} d(x, w^k) \ge d (\tx, \tw)$$
and we mean by $\bra{w^k} \in (y, z)$ that
$$\frac{d(y, \bra{w^k})}{t} = d(y, z) =  \frac{d(\bra{w^k}, z)}{1 - t} \text{ for some } t \in (0, 1).$$

To be precise, Sturm~\cite{St} proved Theorem~\ref{thm:St} for geodesic spaces. The equivalence
of (1) and (2) was established in \cite{BGP} for geodesic spaces. However, it is
easy to see that the original arguments work as well for length spaces. Here, we
briefly recall the proof of the derivation of (3) from (1) in Theorem~\ref{thm:St}.

\begin{proof}[Proof of (1) $\Rightarrow$ (3)]
If $\k > 0$ is positive and $d(p, x_{i_0}) = \pi/\sqrt\k$ for some $i_0 \in I$, inequality
\eqref{ineq:LSS} holds obviously. We may assume that $\bra{x_i \,|\, i \in I}$ is contained in the open
ball $B(p, \pi/\sqrt\k)$. Take a sequence $\bra{x^k_i} \subset J_p$ such that $x^k_i\to x_i$ as $k \to \infty$ for
each $i \in I$. Then
\begin{align*}
\sum_{i, j\in I} \lambda_i\lambda_j \arrowprod{px_i}{px_j}
&= \sum_{i, j\in I} \lambda_i\lambda_j \arrowprod{p\bra{x^k_i}}{p\bra{x^k_j}}\\
&\ge \sum_{i, j\in I} \lambda_i\lambda_j \ip{\bra{\log_{p}{x^k_i}}}{\bra{\log_{p}{x^k_j}}}.
\end{align*}
The inequality is due to the angle monotonicity. Then the desired inequality
follows from the following result of Lang--Schroeder.
\end{proof}

\begin{prop}[Lang--Schroeder~\cite{LS}]\label{prop:LS}
Let $(X, d)$ be an Alexandrov space with
curvature bounded below. Take any finite sequences $\bra{\xi_i}_{i \in I} \subset C_p$ from the tangent
cone at $p \in X$ and $\bra{\lambda_i}_{i \in I} \subset \R_+$. Then
\begin{equation}\label{ineq:LS}
\sum_{i, j\in I} \lambda_i\lambda_j \ip{\xi_i}{\xi_j} \ge 0.
\end{equation}
\end{prop}

Although a neat proof of this result is available in the appendix of the original
paper~\cite{LS}, here we present a sketched proof of Proposition~\ref{prop:LS} for later use.

\begin{proof}
Notice that if $N = 2$, inequality~\eqref{ineq:LS} is a consequence of $a^2 + b^2 \ge 2ab$.

Due to the density of $\Sigma^\prime_p$ in $\Sigma_p$ and the homogeneity of the inner product $\ip{\cdot}{\cdot}$,
we may assume that $\bra{\xi_i \,|\, i\in I} \subset \Sigma^\prime_p \subset C_p$.

Fix $\xi_j := \xi_{i_j}$ and $\lambda_j := \lambda_{i_j}$, $j = 0, 1, 2$ for some $i_0 \in I$ and $i_1, i_2 \in I \setminus \{i_0\}$.
At first, we observe that the inner product is sub-additive.

\begin{sublem}\label{sublem:subad}
For any finite set $Z$ of $\Sigma^\prime_p$ and a positive $\epsilon > 0$, we can find a
direction $\xi_{12} \in \Sigma^\prime_p$ and $\lambda_{12} \ge 0$ such that
\begin{equation}\label{ineq:subad}
\lambda_1 \ip{\xi_1}{\zeta} + \lambda_2 \ip{\xi_2}{\zeta}
\ge \lambda_{12}\ip{\xi_{12}}{\zeta} - \epsilon \text{ for any } \zeta \in Z.
\end{equation}
\end{sublem}

\begin{proof}
Let $\gamma_i : [0, l_i] \to X$ be the constant speed geodesic with $\dot{\gamma}(0) = \lambda_i \cdot\xi_i \in C_p$
for $i = 1, 2$ and $\gamma_\zeta : [0, l_\zeta ] \to X$ be the geodesic with $\dot{\gamma_\zeta}(0) = \zeta$ for each $\zeta \in Z$.
Let $\bra{\epsilon^k}$ be a sequence with $\epsilon^k \to 0+$ as $k \to \infty$. Take $x^k = x^k_{12} \in J_p$ such that
$$\lim_{k\to\infty} d^k(\gamma_1(2\epsilon^k), x^k)
= \lim_{k\to\infty} d^k(\gamma_2(2\epsilon^k), x^k)
= \lim_{k\to\infty} \frac{1}{2} d^k(\gamma_1(2\epsilon^k), \gamma_2(2\epsilon^k)).$$
Here $d^k(p, q) := (\epsilon^k)^{-1}d(p, q)$ is the enlarged distance.

Then it follows that
$$\lim_{k\to\infty}d^k(p, x^k) = \lambda_{12} :=
\sqrt{\lambda_1^2 + \lambda_2^2 + 2\lambda_1\lambda_2 \ip{\xi_1}{\xi_2}}$$
(cf. Sublemma~\ref{sublem} below) and, by the triangle comparison,
$$\lambda_1 \ip{\xi_1}{\zeta} + \lambda_2 \ip{\xi_2}{\zeta} 
\ge \limsup_{k\to\infty}\ (\epsilon^k)^{-2} \arrowprod{px^k}{p\gamma_\zeta(\epsilon^k)}
\ge \lambda_{12} \ip{\barrow{x^k}{p}}{\zeta}.$$
Then letting $\xi_{12} := \arrow{x^k}{p}$ for some large $k \gg 1$ establishes inequality~\eqref{ineq:subad}.
\end{proof}

Letting $J := I \cup \{i_{12}\} \setminus \{i_1, i_2\}$ and $\xi_{i_{12}} := \xi_{12}$
and $\lambda_{i_{12}} := \lambda_{12}$ be the ones
obtained in the sublemma, we have
$$\sum_{i, j\in I} \lambda_i\lambda_j \ip{\xi_i}{\xi_j}
\ge \sum_{i, j\in J} \lambda_i\lambda_j \ip{\xi_i}{\xi_j} - \epsilon.$$
By repeating this process, we obtain a direction $\cxi_0 \in \Sigma^\prime_p$ and $\clambda_0 \ge 0$ such that
$$\sum_{i, j\in I} \lambda_i\lambda_j \ip{\xi_i}{\xi_j}
\ge \lambda_0^2 + \clambda_0^2 + 2\lambda_0\clambda_0 \ip{\xi_0}{\cxi_0} - \epsilon
\ge - \epsilon.$$
Since $\epsilon > 0$ is arbitrary, this concludes the proof of the proposition.
\end{proof}

\begin{rem}
Notice that if $\# I = 2$, inequality~\eqref{ineq:LSS} is just a Cauchy--Schwartz
inequality for the inner product $\arrowprod{xy}{xz}$. Hence what is essential in the implication
(3) $\Rightarrow$ (2) in Theorem~\ref{thm:St} is the case $\# I \ge 3$. To be honest, Part (2) follows
from (3) for any $\bra{x_i}_{i\in I}$ with $\# I = 3$ in general metric spaces.
\end{rem}

\begin{rem}
As was noted in \cite{St}, inequality~\eqref{ineq:LSS} is equivalent to
\begin{align}\label{ineq:St}
\sum_{i, j\in I} \lambda_i\lambda_jd(x_i, x_j)^2
&\le 2\sum_{i \in I} \lambda_id(p, x_i)^2
&\text{ if } \k = 0;\\
\sum_{i, j\in I}\lambda_i\lambda_jC_\k(d(x_i, x_j))
&\le \[\sum_{i \in I}\lambda_iC_\k(d(p, x_i))\]^2
&\text{ if } \k < 0.\nonumber
\end{align}
These inequalities were utilized by Sturm~\cite{St2} in studying an approximation of
energy functional of maps whose target spaces have curvature bounded below.
Inequality~\eqref{ineq:St} was also used by Ohta--Pichot~\cite{OP} who gave a new characterization
of the non-negativity of the curvature in terms of the Markov type of metric
spaces.
\end{rem}

\begin{rem}
In \cite{BN}, Berg--Nikolaev gave the following definitions: for four points
$A, B, C, D$ of a metric space $(X, d)$,
$$\cosq{AB}{CD} :=
\frac{d(A,D)^2 + d(B,C)^2 - d(A,C)^2 - d(B,D)^2}{2d(A,B)d(C,D)}$$
and
\begin{equation}\label{eq:BNprod}
\arrowprodnone{AB}{CD} := d(A,B)d(C,D)\cosq{AB}{CD}.
\end{equation}

One of their main results in \cite{BN} is:
\begin{thm}[Berg--Nikolaev \cite{BN}, cf. Sato \cite{Sa}]\label{thm:BG}
Let $(X, d)$ be a geodesic space.
Then the following are equivalent.
\begin{enumerate}
\item
For any distinct four points $A, B, C, D \in X$,
$\cosq{AB}{CD} \le 1$.
\item
$(X, d)$ is a ${\rm CAT}(0)$ space.
\end{enumerate}
\end{thm}
Hence it is not possible to generalize Theorem~\ref{thm:St} to the inner product of Berg--
Nikolaev given in~\eqref{eq:BNprod} for a general Alexandrov space of non-negative curvature
unless it is flat.
\end{rem}

For a metric space $(X, d)$, we let $P_1(X)$ be the set of all Borel probability
measures on $X$ with finite moment and separable support. We say that a Borel
measure $\mu$ has finite moment when
$\int_X d(p, x)\,d\mu(x)$ is finite for some (and all)
$p \in X$. The support $\supp\mu$ of a Borel measure $\mu$ on $X$ is, by definition,
$$\supp\mu := \bra{ x \in X \bigm| \mu(B(x, \delta)) > 0 \text{ for all } \delta > 0}.$$

It is straightforward to generalize inequality~\eqref{ineq:LSS} to the following form.

\begin{prop}[Sturm \cite{St}]\label{prop:St}
Let $(X, d)$ be an Alexandrov spaces with curvature
$\ge \k$. Then for any probability measure $\mu \in P_1(X)$ and a point $p \in X$, we have
$$\int_X \int_X \arrowprod{px}{py} d\mu(x)d\mu(y) \ge 0.$$
\end{prop}

\section{Statement of the main theorem}\label{sect:4}
In this section, we formulate the main theorems of the present paper.
\begin{mthm}\label{thm:A}
Let $(X, d)$ be an Alexandrov space with curvature $\ge\k$. Suppose
that we have a point $p \in X$ and a probability measure $\mu \in P_1(X)$ such that
\begin{equation}
\int_X\int_X \arrowprod{px}{py} d\mu(x)d\mu(y) = 0.
\end{equation}
Let $Y := \supp \mu \cap B(p, \pi/\sqrt\k)$.
Then $[p, Y]$ can be isometrically embedded into
the model space $M_\k$ and $\conv [p, Y]$ is isometric to that of the embedded image in
$M_\k$.
\end{mthm}

In the statement above, $\conv D$ for any subset $D \subset X$ is the subset of $X$
defined inductively as follows (cf. \cite{LS}): First, we define the subset $H_1(E)$
for any subset $E \subset X$ by
$x \in H_1(E)$ if ands only if $x \in E$ or $x \in pq$ for a \textit{unique} geodesic segment $pq$
connecting their end points $p, q$ in $E$.
Then we set
$$H_0 := D, \qquad H_n := H_1(H_{n-1}) \text{ for } n \ge 1,$$
and $\conv D$ is the closure of the convex hull
$\cup^\infty_{n=0} H_n$ of $D$.

The following is a restatement of Theorem~\ref{thm:A} for $\mu \in P_1(X)$ with finite support.
\begin{mthm}\label{thm:B}
Let $(X, d)$ be an Alexandrov space with curvature $\ge \k$. Suppose
that we have a point $p \in X$ and finite sequences $\bra{x_i}_{i\in I} \subset X$ of points and
$\bra{\lambda_i}_{i\in I} \subset \R_+$ such that
\begin{equation}
\sum_{i, j\in I} \lambda_i\lambda_j \arrowprod{px_i}{px_j} = 0.
\end{equation}
Then $[p, \bra{x_i \,|\, i\in I}]$ can be isometrically embedded into the model space $M_\k$ and
$\conv [p, \bra{x_i \,|\, i\in I}]$ is isometric to that of the embedded image in $M_\k$.
\end{mthm}

We can rewrite Theorem~\ref{thm:B} for $\bra{x_i}_{i\in I}$ with $\# I = 3$ as follows with no difficulty.

\begin{mthm}\label{thm:C}
Let $(X, d)$ be an Alexandrov space with curvature $\ge \k$. Suppose
that we have a quadruple $Y := (x; y, z, w)$ in $X$, with size less than $2\pi/\sqrt\k$ if
$\k > 0$, such that $x \notin [y, z] \cup [z,w] \cup [w, y]$ and
\begin{equation}
\cangle(x; y, z) + \cangle(x; z,w) + \cangle(x;w, y) = 2\pi.
\end{equation}
Then the triangle $\Delta yzw$ spans the unique closed triangular region $D$ which contains
$x$ and isometric to the triangular region in $M^2_\k$ with vertices $\ty, \tz$ and $\tw$.
\end{mthm}

\begin{mcor}\label{cor:D}
Let $(X, d)$ be an Alexandrov space with curvature $\ge \k$. Suppose
that there exist a quadruple $T := \{x, y, z, w\}$ in $X$ with $w \in (y, z)$ and
${\rm peri}(x, y, z) < 2\pi/\sqrt\k$ if $\k > 0$, and an isometric embedding of $T$ into $M^2_\k$; this especially 
means that $d(x,w) = d(\tx, \tw)$. In addition, we assume that we have a point
$p \in (x,w)$. Then the same conclusion as in Theorem~\ref{thm:C} holds for the quadruple
$(p; x, y, z)$.
\end{mcor}

\begin{mcor}\label{cor:E}
Let $(X, d)$ be an Alexandrov space with curvature $\ge \k$. Suppose
that there exist two geodesics $\gamma : [0, l_1] \to X$ and $\eta : [0, l_2] \to X$ emanating from
the same point $o := \gamma(0) = \eta(0)$ of $X$, with ${\rm peri}(o, \gamma(l_1), \eta(l_2)) < 2\pi/\sqrt\k$ if $\k > 0$,
such that
\begin{equation}
\angle(\gamma, \eta) = \cangle(o; \gamma(l_1), \eta(l_2)).
\end{equation}
In addition, we assume that we have a point $p \in (\gamma(l^\prime_1), \eta(l^\prime_2))$ for some $l^\prime_i \in (0, l_i)$,
$i = 1, 2$. Then the same conclusion as in Theorem~\ref{thm:C} holds for the quadruple
$(p; o, \gamma(l_1), \eta(l_2))$.
\end{mcor}

Corollaries~\ref{cor:D} and \ref{cor:E} have been shown already for finite dimensional Alexandrov
spaces by Grove--Markvorsen in the appendix of their paper~\cite{GM}. Theorems~\ref{thm:A}--
\ref{thm:C} are not found in the literature even for the finite dimensional case. We should
mention Shiohama's lecture note~\cite{Sh} as well, where Corollary~\ref{cor:E} were proven under
additional assumptions.

Of course, Theorem~\ref{thm:A} includes Theorems~\ref{thm:B} and \ref{thm:C}. However, we provide the
proofs to each of Theorems~\ref{thm:C}, \ref{thm:B} and \ref{thm:A} in Sections~\ref{sect:6}, ~\ref{sect:7} and \ref{sect:8}, respectively. This
is for the sake of those who wish to understand the proof of Theorems~\ref{thm:B} or \ref{thm:C}
only. This will also help the reader follow the argument in the proof of the most
general Theorem~\ref{thm:A}.

\section{Preliminaries}\label{sect:5}
In this section, we collect several preliminary results, which will be used in the
proofs of Theorems~\ref{thm:A}--\ref{thm:C} later. For those who wants to understand the proof of
Theorem~\ref{thm:C}, only the following lemma will be required.

\begin{lem}\label{lem:connect}
Under the assumption of Theorem~\ref{thm:B}, there exist a unique geodesic
joining $p$ to $x_i$ and a unique antipode of the direction $\arrow{x_i}{p} \in \Sigma_p$ for each $i \in I$.
\end{lem}
Here, by the antipode of $\xi \in \Sigma_p$, we mean the element of $\Sigma_p$, denoted by $-\xi$,
such that $\angle(\xi,-\xi) = \pi$. If the antipode exists, it must be unique for metric
spaces with lower curvature bound.

\begin{proof}
Replacing $\lambda_i d(p, x_i)$ with $\lambda_i$, we may assume that
\begin{equation}
\sum_{i, j\in I} \lambda_i\lambda_j \cos \cangle(p; x_i, x_j) = 0.
\end{equation}
Fix $x_0 := x_{i_0}$ and $\lambda_0 := \lambda_{i_0}$ for some $i_0 \in I$. Take a sequence
$\bra{x^k}$ from $J_p$ such
that $x^k \to x_0$ as $k \to \infty$. We put $x^\prime_{i_0} := x^k$ for some large $k \gg 1$.

For any small $\epsilon > 0$, find a large $K > 0$ such that $k \ge K$ implies that
\begin{equation}
(1 - \epsilon)\lambda_0^2(1 - \cos \epsilon) \ge
\sum_{i, j\in I} \lambda_i\lambda_j \cos \cangle(p; x^\prime_i, x^\prime_j)
\end{equation}
for some point $x^\prime_i$ of $J_p$ near $x_i$ for each $i \ne i_0$.

By Sublemma~\ref{sublem:subad}, we can find $\clambda_0 > 0$ and $\cx \in J_p$ so that
\begin{align*}
\sum_{i, j\in I} \lambda_i\lambda_j \cos \cangle(p; x^\prime_i, x^\prime_j)
&\ge \sum_{i, j\in I} \lambda_i\lambda_j \ip{\arrow{x^\prime_i}{p}}{\arrow{x^\prime_j}{p}}\\
&\ge \lambda_0^2 + \clambda_0^2 + 2\lambda_0\clambda_0 \ip{\arrow{z}{p}}{\arrow{\cx}{p}}
- \epsilon\lambda_0^2(1 - \cos\epsilon)\\
&\ge \lambda_0^2(1 + \ip{\arrow{z}{p}}{\arrow{\cx}{p}}) - \epsilon\lambda_0^2(1 - \cos\epsilon)
\end{align*}
for any $z \in \bra{x^k, x^l}$ with $k, l > K$. We used that
$\ip{\arrow{z}{p}}{\arrow{\cx}{p}}$ is non-positive.

This yields that $\angle(\arrow{z}{p}, \arrow{\cx}{p}) \ge \pi - \epsilon$ and hence
$$\angle\(\arrow{x^k}{p}, \arrow{x^l}{p}\)
\le 2\pi - \angle\(\arrow{x^k}{p}, \arrow{\cx}{p}\) - \angle\(\arrow{x^l}{p}, \arrow{\cx}{p}\)
\le 2\epsilon$$
for $k, l > K$, which implies that $\barrow{x^k}{p}$ is a Cauchy sequence in $(\Sigma_p, \angle)$.

Therefore, there exist a unique geodesic $\gamma_i : [0, l_i] \to X$ from $p$ to $x_i$ and the
antipode of the direction $\arrow{x_i}{p} \in \Sigma_p$, denoted by $-\arrow{x_i}{p}$, for each $i \in I$.
\end{proof}

The isometric embedding in the proofs of Theorems~\ref{thm:A} and \ref{thm:B} relies on the
following lemma (cf.~\cite{Mi}).
\begin{lem}\label{lem:embed}
Let $(\Sigma_p, \angle)$ be a space of directions at a point $p$ of an Alexandrov
space with lower curvature bound. Suppose that we have a subset $Z$ of $\Sigma_p$ such
that for any $\xi \in Z$, we can find its antipode $-\xi \in Z$. Then there exist a Hilbert
space $\H$ and an isometric embedding of $Z$ into the unit sphere of $\H$ equipped with
the angle metric.
\end{lem}

\begin{proof}
Since $\angle(\xi, \eta) = \pi - \angle(-\xi, \eta)$ for any $\xi \in Z$ and $\eta \in \Sigma_p$, it follows that for
any finite sequences $\bra{\xi_i}_{i \in I} \subset Z$ and $\bra{\lambda_i}_{i \in I} \subset \R$,
$$\sum_{i, j\in I} \lambda_i\lambda_j \ip{\xi_i}{\xi_j}
= \sum_{i, j\in I} |\lambda_i| |\lambda_j| \ip{{\rm sgn}(\lambda_i)\xi_i}{{\rm sgn}(\lambda_j)\xi_j}
\ge 0,$$
that is, $\ip{\cdot}{\cdot}$ is a \textit{kernel of positive type} on $Z$. Then the construction of the
isometric embedding of $Z$ is done by a standard method (e.g.~\cite{BL}) which we
now explain.

Consider the vector space $\H^\prime$ consisting of all functions $f : Z \to \R$ with finite
supports. We equip $\H^\prime$ with the inner product given by
$$\ip{f}{g} := \sum_{\xi, \eta\in Z} f(\xi)g(\eta) \ip{\xi}{\eta} \text{ for } f, g \in \H^\prime,$$
and set $\H_0 := \bra{f \in H^\prime \bigm| \ip{f}{f} = 0}$. The Cauchy--Schwartz inequality implies
that $\H_0$ is a vector space.

We define the Hilbert space $(\H, \ip{\cdot}{\cdot})$ as the completion of $(\H^\prime/\H_0, \ip{\cdot}{\cdot})$. Now
$\H$ and the map assigning to each $\xi \in Z$ the characteristic function of the one
point set $\{\xi\}$ are what we are seeking for.
\end{proof}

Next, let us consider the situation where we have two geodesics $\gamma_i : [0, l_i] \to X, i = 1, 2$,
starting the same point $p$ of $X$ such that
\begin{equation}
\cangle(p; \gamma_1(l_1), \gamma_2(l_2)) = \angle(\gamma_1, \gamma_2).
\end{equation}
Here we intend to prove a sort of rigidity lemma (cf. Corollary~\ref{cor:E}), which we will need
in the proof of Theorem~\ref{thm:A}.

\begin{lem}\label{lem:rigid}
Let $x^\prime_i := \gamma_i(l^\prime_i)$ for some $0 < l^\prime_i < l_i$, $i = 1, 2$. Take a sequence $\bra{x^k}$ 
in $J_p$ such that $\bra{x^k} \in (x^\prime_1, x^\prime_2)$ and geodesics $\eta^k$ from $p$ to $x^k$.
Then for any $\epsilon > 0$
and $i = 1, 2$, we have
$$\cangle\(p; x_i, \bra{x^k}\)
= \cangle\(p; \gamma_i(\epsilon), \bra{\eta^k(\epsilon)}\).$$
In particular, if we have a direction $\arrow{x}{p}$ in $(\Sigma^\prime_p, \angle)$
with $x \in (x^\prime_1, x^\prime_2)$, then
$$\cangle(p; x_i, x) = \angle(\arrow{x_i}{p}, \arrow{x}{p}) \qquad\text{ for  } i = 1, 2.$$
\end{lem}

\begin{proof}
It suffices to show that LHS $\ge$ RHS for $i = 1$. We employ the technique
which was used by Plaut in the proof of \cite[Proposition 51]{Pl:Surv}. Fix a small positive
number $\epsilon^\prime \ll \epsilon$. We first verify the following.

\begin{sublem}\label{sublem}
Let $\tx \in (\tx^\prime_1, \tx^\prime_2)$ and $\widetilde{\eta(\epsilon^\prime)}$ be the corresponding points of $M^2_\k$.
Then we have
\begin{align}
&d(p, \bra{x^k}) = d(\tp, \tx) \quad\text{ and }\quad d(x_1, \bra{x^k}) = d(\tx_1, \tx).&\\
&d(x^\prime_1, \bra{\eta^k(\epsilon^\prime)}) = d(\tx^\prime_1, \widetilde{\eta(\epsilon^\prime)}).&
\end{align}
\end{sublem}

\begin{proof}
Indeed, by the angle monotonicity and the quadruple condition,
\begin{align*}
2\pi
&= \cangle(x^\prime_1; x^\prime_2, x_1) + \cangle(x^\prime_1; x_1, p) + \cangle (x^\prime_1; p, x^\prime_2)\\
&\le \cangle(x^\prime_1; \bra{x^k}, x_1) + \cangle(x^\prime_1; x_1, p) + \cangle(x^\prime_1; p, \bra{x^k})\\
&\le 2\pi.
\end{align*}
These inequalities must be equalities and the first equation follows.

To see the second equation, we use that
\begin{align*}
2\pi
&= \cangle(\bra{x^k}; x^\prime_1, p) + \cangle(\bra{x^k}; p, x^\prime_2) + \cangle(\bra{x^k}; x^\prime_2, x^\prime_1)\\
&\le \cangle(\bra{x^k}; x^\prime_1, \bra{\eta^k(\epsilon^\prime)})
+ \cangle(\bra{x^k}; \bra{\eta^k(\epsilon^\prime)}, x^\prime_2)
+ \cangle(\bra{x^k}; x^\prime_2, x^\prime_1)\\
&\le 2\pi,
\end{align*}
from which the desired equation follows.
\end{proof}

Fix large $k \gg 1$ and take $\bra{z^{k, l}} \subset J_{x^\prime_1}$ such that $z^{k, l} \to \eta^k(\epsilon^\prime)$ as $l \to \infty$ and
let $\sigma^l = \sigma^{k, l}$ be the geodesic from $z^{k, l}$ to $x^\prime_1$.
Note that $\sigma^{k, l}(\epsilon)$ approaches to $\gamma_1(\epsilon)$
as $k, l \to \infty$ and $\epsilon^\prime \to 0+0$. From this, we have
\begin{align*}
\cangle(p; \gamma_1(\epsilon), \eta^k(\epsilon)) - \tau(k, \epsilon^\prime)
&\le \cangle(p; \bra{\sigma^l(\epsilon)}, \eta^k(\epsilon))\\
&\le \cangle(p; \bra{\sigma^l(\epsilon)}, \eta^k(\epsilon^\prime))\\
&\le \cangle(p; x^\prime_1, \eta^k(\epsilon^\prime)).
\end{align*}

Letting $k \to \infty$ yields that
\begin{align*}
\cangle(p; \gamma_1(\epsilon), \bra{\eta^k(\epsilon)}) - \tau(\epsilon^\prime)
&\le \cangle(p; x^\prime_1, \bra{\eta^k(\epsilon^\prime)})\\
&= \cangle(p; x_1, \bra{x^k}).
\end{align*}
Since $\epsilon^\prime > 0$ is taken arbitrarily, this concludes the proof.
\end{proof}

If the Alexandrov space $X$ happens to be of finite dimension, then the space
of directions at any point is known to be $({\rm dim}X - 1)$-dimensional Alexandrov
space with curvature $\ge 1$ (\cite{BGP}). This fact played a crucial role in the proof of
the finite dimensional case of Corollaries~\ref{cor:D} and \ref{cor:E} in \cite[Appendix]{GM}.

However, this is not the case in the infinite dimensional case; a counterexample
showing that a space of directions at a point of an infinite dimensional Alexandrov
space may fail to be a length space is constructed by Halbeisen~\cite{Ha}. In spite
of this, we have the following proposition, which will turn out to be enough for
our purpose. Recall the definition of a metric space having curvature $\ge \k$ from
Definition~\ref{def:curv}.

\begin{prop}\label{prop:Sigma}
Let $(X, d)$ be an Alexandrov space with curvature $\ge \k$ and
$p \in X$ be a point. Then
\begin{enumerate}
\item
the tangent cone $(C_p, |\cdot|)$ at $p$ has curvature $\ge 0$, and
\item
the space of directions $(\Sigma_p, \angle)$ at $p$ has curvature $\ge 1$.
\end{enumerate}
\end{prop}

Part (1) follows from the definition almost immediately (cf.~\cite{Mi}). We give only
the proof of Part (2), which is more subtle in contrast to Part (1). This answers
Plaut's question; see the paragraph following Proposition 49 of \cite{Pl:Surv}.

\begin{proof}[Proof of Part (2)]
Fix any four points $\xi_i, i = 0, \dots, 3$, in $\Sigma^\prime_p$ and let $\xi_4 := \xi_1$. For
each $i = 0, \dots, 4$, we let $\gamma_i : [0, l_i] \to X$ be the unit speed geodesic representing
$\xi_i$ and find $\txi_i$ in the unit sphere $S$ of $\R^3$ such that
$$\theta_i := \angle(\txi_0, \txi_i) = \angle(\xi_0, \xi_i)
\quad\text{ and }\quad
\onegle(\txi_0; \txi_i, \txi_{i+1}) = \onegle(\xi_0; \xi_i, \xi_{i+1})$$
for $i = 1, 2, 3$.

Consider the plane $P$ in $\R^3$ tangent to $S$ at $\txi_0$.
For each $i = 1, \dots, 4$, let $\teta_i$ be the point on $P$ determined by
$$\bra{\teta_i} =
\begin{cases}
[2\txi_0, 2\txi_i] \cap P &\text{ for } i \text{ with } \theta_i > \pi/3;\\
[0, 2\txi_i] \cap P &\text{ for } i \text{ with } \theta_i \le \pi/3,
\end{cases}$$
and find a sequence $\bra{x^k_i}$ in $X$ such that
$$\lim_{k\to\infty} d^k(\gamma_0(2\epsilon^k), x^k_i) = \norm{2\txi_0 - \teta_i}
\text{ and }
\lim_{k\to\infty} d^k(\gamma_i(2\epsilon^k), x^k_i) = \norm{2\txi_i - \teta_i}.$$
Here, $d^k := (\epsilon^k)^{-1}d$ is the enlarged distance for some $\epsilon^k \to 0+$ as $k \to \infty$.

Then, by the Euclidean geometry,
$$\onegle(\xi_0; \xi_i, \xi_{i+1})
= \onegle(\txi_0; \txi_i, \txi_{i+1})
= \zengle(\txi_0; \teta_i, \teta_{i+1}),$$
and by the triangle comparison,
$$\sum^3_{i=1} \zengle\(\txi_0; \teta_i, \teta_{i+1}\)
\le \sum^3_{i=1} \zengle\(\bra{\gamma_0(\epsilon^k)}; \bra{x^k_i}, \bra{x^k_{i+1}}\)
\le 2\pi.$$
In deriving the first inequality, we used the technique used in the proof of the
previous proposition to get
$$\lim_{k\to\infty} d^k(\gamma_0(\epsilon^k), x^k_i)
= \norm{\txi_0 - \teta_i} \text{ for each } i = 1, \dots, 4.$$
The second inequality follows from that $(X, d^k)$ has curvature $\ge (\epsilon^k)^2\k \to 0$ as
$k \to \infty.$ This completes the proof.
\end{proof}

\section{Proof of Theorem~\ref{thm:C}}\label{sect:6}
In this section, we provide a proof of Theorem~\ref{thm:C}. As mentioned above, it is a
special case of Theorem~\ref{thm:B}. In spite of this, we decided to give an independent
proof, because the structures of their proofs are somewhat different form each
other. The reader may wish to skip this section.

At first, we establish the following technical lemma (cf. Corollary~\ref{cor:D}).
\begin{lem}\label{lem:tech}
Let $(X, d)$ be an Alexandrov space with curvature $\ge \k$. Suppose that
we have a quadruple $Y := (x; y, z, w)$ in $X$, with size is less than $2\pi/\sqrt\k$ if $\k > 0$,
such that
\begin{equation}
\cangle(x; y, z) + \cangle(x; z,w) + \cangle(x;w, y) = 2\pi.
\end{equation}

Assume in addition that
$\cangle(x; y, z) \cdot \cangle(x; y,w) > 0$ and that there exist geodesics $xy, xz, xw$ connecting $x$
and $y, z,w$, respectively.
Let $(\tx; \ty, \tz, \tw)$ be the corresponding quadruple in $M^2_\k$.

Then for any $p \in (x, y)$ and $q \in xz\cup xw$, there exists a unique geodesic segment
$pq$ connecting them and
$$D_y := \text{the closure of }  [xy \setminus \{x, y\}, xz \cup xw]$$
is isometric to the union of two triangular regions in $M^2_\k$ determined by $\Delta\tx\ty\tz$
and $\Delta\tx\ty\tw$.
\end{lem}

Before proceeding to the proof of this technical lemma,
we explain how Theorem~\ref{thm:C} follows from Lemma~\ref{lem:tech}.

\begin{proof}[Proof of Theorem~\ref{thm:C}]
At first, find an isometric embedding map $F_0 : Y \to M^2_\k$.
By applying Lemma~\ref{lem:connect}, we can find unique geodesics $xy$, $xz$, $xw$ from $x$ to $y, z, w$,
respectively. Using Lemma~\ref{lem:tech}, we extend the map $F_0$ to the isometric embedding
$F_{x^\prime} : D_{x^\prime} \to M^2_\k$ for each $x^\prime \in \{y, z,w\}$. Let $D :=
\bigcup\bra{D_{x^\prime} \,|\, x^\prime \in \{y, z,w\}}$. It is
clear that the natural map $F : D \to M^2_\k$ determined by $F|_{D_{x^\prime}} = F_{x^\prime}$ is gives the
isometry between $D$ and $F(D)$. Now the proof of Theorem~\ref{thm:C} is complete.
\end{proof}

\begin{proof}[Proof of Lemma~\ref{lem:tech}]
The proof is divided into several steps. First find an isometric
embedding map $F_0$ of $Y := (x; y, z, w)$ into the model surface $M^2_\k$.

\begin{step}
The trivial extension $F: xy\cup xz\cup xw \to M^2_\k$ of $F_0$ is a distance-preserving
map.
\end{step}

\begin{proof}
Take any three points $q_{x^\prime} \in xx^\prime \setminus \{x\}$ for each $x^\prime \in \{y, z, w\}$. Then, by the
angle monotonicity,
\begin{align*}
2\pi
&= \cangle(x; y, z) + \cangle(x; z, w) + \cangle(x; w, y)\\
&\le \cangle(x; q_y, q_z) + \cangle(x; q_z, q_w) + \cangle(x; q_w, q_y)\\
&\le 2\pi.
\end{align*}
Hence $\cangle (x; x^\prime, x^{\prime\prime}) = \cangle (x; qx^\prime, qx^{\prime\prime})$, or equivalently,
$d(q_{x^\prime}, q_{x^{\prime\prime}}) = d(F(q_{x^\prime}), F(q_{x^{\prime\prime}}))$ for any $x^\prime, x^{\prime\prime} \in \{y, z, w\}$.
\end{proof}

\begin{step}
Take points $p \in (x, y)$, $q_1 \in xz$ and $q_2 \in xw$. Then for any sequences
$\bra{q^k_1}$ and $\bra{q^l_2}$ in $J_p$ converging to $q_1$ and $q_2$, respectively, we have
$$\angle\(\barrow{q^k_1}{p}, \barrow{q^l_2}{p}\)
= \cangle(p; q_1, q_2).$$
\end{step}

\begin{proof}
As in the previous step,
\begin{align*}
2\pi
&= \cangle(p; q_1, q_2) + \cangle(p; q_1, y) + \cangle(p; q_2, y)\\
&= \cangle(p; \bra{q^k_1}, \bra{q^l_2}) + \cangle(p; \bra{q^k_1}, y) + \cangle(p; \bra{q^l_2}, y)\\
&\le \angle\(\barrow{q^k_1}{p}, \barrow{q^l_2}{p}\) + \angle\(\barrow{q^k_1}{p}, \arrow{y}{p}\)
+ \angle\(\barrow{q^l_2}{p}, \arrow{y}{p}\)\\
&\le 2\pi.
\end{align*}
Then we have equalities in the above and hence the desired equation follows.
\end{proof}

\begin{step}
For any $p \in (x, y)$ and $q \in xz \cup xw$, there exists a unique geodesic from
$p$ to $q$.
\end{step}

\begin{proof}
We may assume that $q_1 := q \in xz \setminus \{x\}$. We take any $q_2 \in xw \setminus \{x\}$ and
sequences
$\bra{q^k_1}$ and $\bra{q^l_2} \subset J_p$ converging to $q_1$ and $q_2$, respectively.
Now we introduce a notation. For two positive numbers $\theta_1, \theta_2 > 0$ with $\theta_1+\theta_2 <
2\pi$, we let
\begin{equation}
S(\theta_1, \theta_2) := \frac{\sin(\theta_1 + \theta_2)}{\sin \theta_1 + \sin \theta_2}.
\end{equation}

Let us explain where the definition of $S(\theta_1, \theta_2)$ comes from. Consider three
points $\xi, \eta_1, \eta_2$ on the unit circle in $\R^2$ with center $0 \in \R^2$ such that
$$\angle(\xi, \eta_i) = \theta_i, i = 1, 2 \qquad\text{ and }\qquad  \angle(\eta_1, \eta_2) = \theta_1 + \theta_2.$$
Then it is easily checked that $S(\theta_1, \theta_2)\xi$ lies on the line through $\eta_1 and \eta_2$.

Now we observe that
$\barrow{q^k_1}{p}$ is a Cauchy sequence in $\Sigma_p$. Indeed,
letting $\alpha := S(\cangle(p; q_1, x), \cangle(p; q_2, x))$,
we use the fact that the tangent cone $(C_p, |\cdot|)$ has
non-negative curvature (cf. Proposition~\ref{prop:Sigma}) to see that
\begin{align*}
&\zengle\(\alpha \cdot \arrow{x}{p}; \barrow{q^k_1}{p}, \barrow{q^l_1}{p}\)\\
&\le 2\pi
- \zengle\(\alpha \cdot \arrow{x}{p}; \barrow{q^k_1}{p}, \barrow{q^l_2}{p}\)
- \zengle\(\alpha \cdot \arrow{x}{p}; \barrow{q^l_1}{p}, \barrow{q^l_2}{p}\)\\
&= 2\pi - \pi - \pi = 0.
\end{align*}
We have used that
$$\norm{\barrow{q^k_1}{p} - \alpha \cdot \arrow{x}{p}} 
+ \norm{\alpha \cdot \arrow{x}{p} - \barrow{q^l_2}{p}} =
\norm{\barrow{q^k_1}{p} - \barrow{q^l_2}{p}}$$
and hence $\zengle\(\alpha \cdot \arrow{x}{p}; \arrow{q^k_1}{p}, \arrow{q^l_2}{p}\)$
goes to $\pi$ as $k, l \to \infty$.

By the angle monotonicity and completeness of $(X, d)$, the sequence
$\bra{pq^k}$ of geodesics converges to the geodesic $pq$.
The uniqueness of the direction $\arrow{q}{p} \in \Sigma_p$
follows from the construction.
\end{proof}

\begin{step}
For any $p \in (x, y)$ and any $q_1, q_2 \in xz \cup xw$,
$$\angle(\arrow{q_1}{p}, \arrow{q_2}{p})
= \cangle(p; q_1, q_2).$$
\end{step}

\begin{proof}
In the case where $q_1 \in xz$ and $q_2 \in xw$, the assertion follows from Step 2.
We may assume that $q_1 \in (x, q_2) \cap xz$. Set
$$\beta_0 := S(\cangle(p; x, w), \cangle(p; x, q_2))
\text{ and }
\beta_1 := S(\cangle(p; q_1, w), \cangle(p; q_1, q_2)).$$

In this case, it follows from the same argument as in the previous step that
$$\zengle(\beta_0 \cdot \arrow{x}{p}; \beta_1 \cdot \arrow{q_1}{p}, \arrow{q_2}{p}) = 0,$$
that is,
\begin{align*}
\angle(\arrow{q_1}{p}, \arrow{q_2}{p})
&= \angle(\arrow{q_2}{p}, \arrow{x}{p}) - \angle(\arrow{q_1}{p}, \arrow{x}{p})\\
&= \cangle(p; q_2, x) - \cangle(p; q_1, x)
= \cangle(p; q_1, q_2).
\end{align*}
This is the desired equation.
\end{proof}

Finally, we construct the isometric embedding map $F_y : D_y \to M^2_\k$ by extending
$F$ in Step 1. For any $p \in (x, y)$, we let $D_p \subset X$ be the set of points on a geodesic
from $p$ to some $q \in xz \cup xw$. We also define the map $F_p : D_p \to M^2_\k$ naturally.
Due to the assertion in Step 4, the map $F_p$ must be distance preserving.
Due to the uniqueness part of Step 3, it is easy to see that
$$D_{p_2} \subset D_{p_1}
\qquad\text{ and }\qquad
F_{p_1} \equiv F_{p_2} \text{ on } D_{p_2}$$
for any $p_1, p_2 \in (x, y)$ with $p_1 \in (y, p_2)$. Noting that $D_y$ is the closure of
$\bigcup \bra{D_p \,|\, p \in xy \setminus\{x, y\}}$,
we define the map $F_y : D_y \to M^2_\k$ by extending each
$F_p: D_p \to M^2_\k$. This finishes the proof of Lemma~\ref{lem:tech}.
\end{proof}

We leave the proof of Corollaries~\ref{cor:D} and \ref{cor:E} to the reader (cf.~\cite[Appendix]{GM}).

\section{Proof of Theorem~\ref{thm:B}}\label{sect:7}
In this section, we describe the strategy of the proofs of Theorems~\ref{thm:A} and \ref{thm:B},
and prove Theorem~\ref{thm:B}.
In Theorems~\ref{thm:A} and \ref{thm:B}, we want to show that $\conv [p, Y]$ is isometrically embedded
into the model space $M_\k$ for some subset $Y$ of the open ball $B(p, \pi/\sqrt\k) \subset X$,
and the proof which we now give consists of three main steps:
\begin{quotation}
\begin{itemize}
\item[\textbf{Step 1}]
Show that there exist a unique direction $\arrow{x}{p}$ and its antipode $-\arrow{x}{p}$ 
in $\Sigma_p$, and $\angle(\arrow{x}{p}, \arrow{y}{p})
= \cangle(p; x, y)$ for each $x, y \in Y \setminus\{p\}$.
\item[\textbf{Step 2}]
Construct an isometric embedding map of $[p, Y]$ into $M_\k$.
\item[\textbf{Step 3}]
Extend the map in Step 2 to the isometric embedding of $\conv [p, Y]$ into $M_\k$.
\end{itemize}
\end{quotation}

\begin{proof}[Proof of Theorem~\ref{thm:B}]
Recall that we have sequences $\bra{x_i}_{i\in I}$ in $X$ and $\bra{\lambda_i}_{i\in I}$ in
$\R_+$ such that the equality holds in~\eqref{ineq:LSS}. Changing $\lambda_i$'s, we may assume that
\begin{equation}
\sum_{i, j\in I} \lambda_i\lambda_j \cos \cangle(p; x_i, x_j) = 0.
\end{equation}
We know that $\bra{x_i \,|\, i\in I}$ lies in the ball $B(p, \pi/\sqrt\k)$ when $\k > 0$.
We realize that Lemmas~\ref{lem:connect} and \ref{lem:embed} establish the first two steps and only the
final step remains to be done.

Take a small $\delta > 0$. We fix $x_1 := x_{i_1}$ and $x_2 := x_{i_2}$ for some $i_1, i_2 \in I$ and let
$x_{i-\delta} := \gamma_i(l_i - \delta)$ for $i = 1, 2$,
where $\gamma_i : [0, l_i] \to X$ is the geodesic from $p$ to $x_i$.
We may assume that $0 < \cangle(p; x_1, x_2) < \pi$.

Take any sequence
$\bra{x^k} \subset J_p$ such that
$\bra{x^k} \in (x_{1-\delta}, x_{2-\delta})$ and find the
corresponding point $\tx = \tx_{12-\delta}$ of $M_\k$. We observed that
$d(p, \bra{x^k}) = d(\tp, \tx)$ in
Sublemma~\ref{sublem}.

Find $\lambda > 0$ and $u \in (0, 1)$ such that
$$\lambda\arrow{\tx}{\tp} = (1 - u)\arrow{\tx_1}{\tp} + u\arrow{\tx_2}{\tp} \in C_\tp.$$
Recall that the tangent cone of the model space $M_\k$ is a Hilbert space. We prepare
the notations. Assuming $\lambda_{i_1} \le \lambda_{i_2}$, we let $J := I \cup \{i_{12}\}$ and
\begin{align*}
\lambda^\prime_{i_1} &:= \lambda_{i_1}u;&
\lambda^\prime_{i_2} &:= \lambda_{i_2} - \lambda_{i_1}u;\\
\lambda^\prime_{i_{12}} &:= \lambda\lambda_{i_1};&
 \lambda^\prime_i &:= \lambda_i \text{ for } i \in I \setminus\{i_1, i_2\}.
\end{align*}
Then, with $\tx_{i_12}$ being $\tx = \tx_{12-\delta}$,
$$0 =
\sum_{i \in I} \lambda_i\arrow{\tx_i}{\tp} =
\sum_{i \in J} \lambda^\prime_i\arrow{\tx_i}{\tp} \in C_\tp$$
and
$$0 =
\sum_{i, j\in J} \lambda^\prime_i \lambda^\prime_j \ip{\arrow{\tx_i}{\tp}}{\arrow{\tx_j}{\tp}}
\ge \sum_{i, j\in J} \lambda^\prime_i \lambda^\prime_j \ip{\arrow{x_i}{p}}{\arrow{x_j}{p}}
\ge 0.$$
Here we understand these inequalities by setting $\arrow{\tx_i}{\tp} := \arrow{\tx}{\tp}$
and $\arrow{x_i}{p} := \barrow{x^k}{p}$
for $i = i_{12}$. Then the first inequality follows from the triangle comparison, and
Proposition~\ref{prop:LS} implies the last inequality.

Hence we get
$$\angle\(\barrow{x^k}{p}, \arrow{x_i}{p}\)
= \cangle(\tp; \tx, \tx_i) \text{ for each } i \in I.$$
Since the space of directions $(\Sigma_p, \angle)$ has curvature $\ge 1$ (Proposition~\ref{prop:Sigma}),
\begin{align*}
&\onegle\(\arrow{x_1}{p}; \barrow{x^k}{p}, \barrow{x^l}{p}\)\\
&\le 2\pi - \onegle\(\arrow{x_1}{p}; \barrow{x^k}{p}, -\arrow{x_2}{p}\)
- \onegle\(\arrow{x_1}{p}; \barrow{x^l}{p}, -\arrow{x_2}{p}\)\\
&= 0.
\end{align*}

Therefore,$\barrow{x^k}{p}$ and $\bra{x^k}$
are Cauchy sequences, and hence we can find the
unique point $x_{12-\delta} \in (x_{1-\delta}, x_{2-\delta})$ where $\bra{x^k}$
converges to.

Since $\delta > 0$ is arbitrary, there exists a point $x_{i_{12}} = x_{12} \in (x_1, x_2)$ which satisfies
$$\sum_{i, j\in J} \lambda^\prime_i\lambda^\prime_j \cos \cangle(p; x_i, x_j) = 0.$$
This process can be repeated inductively.

Now we define $H^\prime_0 := [p, \bra{x_i \,|\, i \in I}]$ and $H^\prime_n := H^\prime_1 (H^\prime_{n-1})$ for $n \ge 2$, where $H^\prime_1(E)$,
for any subset $E \subset X$, is the set of the point $x$ of $X$ with $x \in E$ or $x = \lim_{k\to\infty} x^k$
for some sequence $\bra{x^k} \in (x^\prime_1, x^\prime_2)$ with $x^\prime_i \in (p, x_i)$ and $x_i \in E$, $i = 1, 2$. Then
$H^\prime_n \subset H_n \subset \overline{H^\prime_n}$.
Moreover, we observed that $H^\prime_n$ is isometrically embedded into
the model space $M_\k$, and the embedding of $H^\prime_n$ is the extension of that of $H^\prime_{n-1}$
for any $n\ge 1$.
Now the proof of Theorem~\ref{thm:B} is complete.
\end{proof}

\section{Proof of Theorem~\ref{thm:A}}\label{sect:8}
Now we are in a position to present the proof of Theorem~\ref{thm:A}. The structure of
the proof is identical to that of Theorem~\ref{thm:B} in the previous section.

\begin{proof}[Proof of Theorem~\ref{thm:A}]
Recall that we have a point $p$ of $X$ and a probability measure
$\mu$ in $P_1(X)$ satisfying
\begin{equation}\label{assume:A}
\int_X \int_X \arrowprod{px}{py} d\mu(x)d\mu(y) = 0.
\end{equation}

The proof is divided into three steps as is that of Theorem~\ref{thm:B}. Since now we
are handling a general probability measure $\mu$, we have to modify the argument in
the proof of Theorem~\ref{thm:B} in establishing the first and the final steps. This is done
in the following two lemmas.

\begin{lem}
Under the assumption of Theorem~\ref{thm:A},
for any point $x$ in $Y := \supp \mu \cap B(p, \pi/\sqrt\k)$ other than $p$,
there exists unique direction $\arrow{x}{p}$
and its antipode $-\arrow{x}{p}$
in $\Sigma_p$, and
$\cangle(p; x, y) = \angle(\arrow{x}{p}, \arrow{y}{p})$
for any $x$ and $y$ in $Y \setminus\{p\}$.
\end{lem}

\begin{proof}
At first, we verify
\begin{clm}
For any point $z$ of $X$, we have
\begin{equation}\label{ineq:ge0}
\phi(z) := \int_X \arrowprod{pz}{px} d\mu(x) \ge 0.
\end{equation}
Moreover, we have the equality in~\eqref{ineq:ge0} provided $z$ is in $Y$.
\end{clm}

Indeed, we apply Proposition~\ref{prop:St} to the probability measure $\frac{1}{1+\epsilon}(\mu + \epsilon\delta_z)$ for
some small $\epsilon > 0$ to obtain
$$\int_X\int_X \arrowprod{px}{py} d\mu(x)d\mu(y)
+ 2\epsilon \int_X \arrowprod{pz}{px} d\mu(x)
+ \epsilon^2 \arrowprod{pz}{pz}
\ge 0.$$
Letting $\epsilon \to 0+$ yields the desired inequality~\eqref{ineq:ge0}. The second assertion follows
from~\eqref{ineq:ge0} and the assumption~\eqref{assume:A}.

Fix two points $x_0$ and $y_0$ of $Y$, a positive $\delta > 0$ and a finite subset $Z$ of $J_p$.
Let $B := B(x_0, \delta)$ and $B^c := X \setminus B(x_0, \delta)$. We see that there exist vectors in $C_p$
approximating $\frac{1}{\mu(B)}\int_{B} \log_{p}x\,d\mu(x)$ and $\frac{1}{\mu(B)}\int_{B^c} \log_{p}x\,d\mu(x)$, respectively, in a sense to
be made clear below.

Recall that any probability measure $\mu$ in $P_1(X)$ is \textit{tight}, that is, for any $\epsilon > 0$,
there is a compact subset $K_\epsilon \subset X$ such that $\mu(K_\epsilon) > 1-\epsilon$, and that we can find a
sequence of probability measures $\bra{\mua}$ with finite supports in $J_p$, i.e., $\mua(F^\alpha) = 1$
for some finite subset $F^\alpha \subset J_p$, such that
\begin{align*}
\int_X\int_X \arrowprod{px}{py} d\mua(x)d\mua(y)
&\to
\int_X\int_X \arrowprod{px}{py} d\mu(x)d\mu(y) = 0;\\
\int_X \arrowprod{pz}{px} d\mua(x)
&\to
\int_X \arrowprod{pz}{px} d\mu(x) = \phi(z) \quad \text{ for any } z \in Z
\end{align*}
as $\alpha \to \infty$ (e.g. Dudley~\cite{Du}).

Fix a large $\alpha \gg 1$.
We use Sublemma~\ref{sublem:subad} to find vectors $\xia$, $\cxia$ and $\etaa$ in $C_p$
approximating $\frac{1}{\mu(B)}\int_{B} \log_{p}x\,d\mua(x)$, $\frac{1}{\mu(B)}\int_{B^c} \log_{p}x\,d\mua(x)$
and $\xia + \cxia$, respectively, in
the sense that:
\begin{align*}
&\frac{1}{\mu(B)^2}
\int_X\int_X \arrowprod{px}{py} d\mua(x)d\mua(y)\\
&=
\frac{1}{\mu(B)^2} \[\int_{B}\int_{B} +2\int_{B}\int_{B^c}+\int_{B^c}\int_{B^c}\]
\arrowprod{px}{py} d\mua(x)d\mua(y)\\
&\ge
\frac{1}{\mu(B)^2} \[\int_{B}\int_{B}+2\int_{B}\int_{B^c}+\int_{B^c}\int_{B^c}\]
\ip{\log_{p}x}{\log_{p}y} d\mua(x)d\mua(y)\\
&\ge
\norm{\xia}^2 + 2 \ip{\xia}{\cxia} + \norm{\cxia}^2 - \alpha^{-1}\\
&= \norm{\etaa}^2 - \alpha^{-1},
\end{align*}
and for any $z \in Z$,
\begin{align*}
\frac{1}{\mu(B)} \int_X \arrowprod{pz}{px} d\mua(x)
&= \frac{1}{\mu(B)}\[\int_{B}+\int_{B^c}\] \arrowprod{pz}{px} d\mua(x)\\
&\ge \frac{1}{\mu(B)}\[\int_{B}+\int_{B^c}\] \ip{\log_{p}z}{\log_{p}x} d\mua(x)\\
&\ge \ip{\log_{p}z}{\xia} + \ip{\log_{p}z}{\cxia} - \alpha^{-1}\\
&\ge \ip{\log_{p}z}{\etaa} - 2\alpha^{-1}.
\end{align*}

Letting $\alpha \to \infty$, we know from these conditions that
$$\norm{\bra{\etaa}}^2 = \norm{\bra{\xia}}^2 + 2 \ip{\bra{\xia}}{\bra{\cxia}} + \norm{\bra{\cxia}}^2= 0;$$
\begin{align*}
\norm{\bra{\xia}}^2
&= \frac{1}{\mu(B)^2} \int_{B}\int_{B} \arrowprod{px}{py} d\mu(x)d\mu(y)\\
&= \frac{1}{\mu(B)^2} \int_{B^c}\int_{B^c} \arrowprod{px}{py} d\mu(x)d\mu(y) =
\norm{\bra{\cxia}}^2;
\end{align*}
and hence, for any $z \in Z$,
\begin{align}\label{ineq:est}
\ip{\log_{p}z}{\bra{\xia}} + \frac{\phi(z)}{\mu(B)}
&\ge
\frac{1}{\mu(B)} \int_{B} \arrowprod{pz}{px} d\mu(x)
&\ge
\ip{\log_{p}z}{\bra{\xia}};\\
\ip{\log_{p}z}{\bra{\cxia}} + \frac{\phi(z)}{\mu(B)}
&\ge
\frac{1}{\mu(B)} \int_{B^c} \arrowprod{pz}{px} d\mu(x)
&\ge
\ip{\log_{p}z}{\bra{\cxia}};\nonumber
\end{align}
and
$$\frac{\phi(z)}{\mu(B)} 
= \frac{1}{\mu(B)} \[\int_{B}+\int_{B^c}\] \arrowprod{pz}{px} d\mu(x)
\ge \ip{\log_{p}z}{\bra{\etaa}}
= 0.$$

Now we verify
\begin{clm}
For any sequence $\bra{x^k}$ in $J_p$ converging to $x_0$ as $k \to \infty$, the sequence
$\barrow{x^k}{p}$
is a Cauchy sequence in $(\Sigma_p, \angle)$.
\end{clm}

To see this, we find $\bra{\xia}$ and $\bra{\cxia}$
for a two point set $Z :=\bra{x^k, x^l}$
with
$k, l \gg \delta-1 \gg 1$. Then it follows that
\begin{align*}
\cos\angle(\arrow{z}{p}, \bra{\cxia})
&\le \frac{1}{\mu(B)} \frac{\int_{B^c} \arrowprod{pz}{px} d\mu(x)}{d(p, z)\norm{\bra{\cxia}}}\\
&=
\frac{1}{\mu(B)} \frac{\phi(z) - \int_{B} \arrowprod{pz}{px} d\mu(x)}{d(p, z) \norm{\bra{\cxia}}}
\to -1
\end{align*}
for $z \in \bra{x^k, x^l}$ as $k, l \to \infty$ and $\delta \to 0+$.

Since the space of directions $(\Sigma_p, \angle)$ has curvature $\ge 1$ (Proposition~\ref{prop:Sigma}),
\begin{align*}
\angle\(\arrow{x^k}{p}, \arrow{x^l}{p}\)
&\le 2\pi - \angle\(\arrow{x^k}{p}, \bra{\cxia}\)
+ \angle\(\arrow{x^l}{p}, \bra{\cxia}\)\\
&\to 2\pi - \pi - \pi = 0 \qquad \text{ as } k, l \to \infty \text{ and } \delta \to 0+.
\end{align*}
Then the Cauchy sequence $\barrow{x^k}{p}$
converges to the direction $\arrow{x_0}{p}$ in $\Sigma_p$. It also
follows that the antipode of $\arrow{x_0}{p}$ exists and is unique.

Finally, we have to verify
\begin{equation}\label{eq:cangle=angle}
\cangle(p; x_0, y_0) = \angle(\arrow{x_0}{p}, \arrow{y_0}{p}).
\end{equation}

To do this, find $\bra{\xia}$ and $\bra{\cxia}$
for $Z := \{x_0, y_0\}$. By letting $\delta$ tend to $0+$, we can
make $\bra{\xia}$ arbitrarily close to $\log_{p}x_0$.
Then, by using that $\phi(y_0) = 0$ and~\eqref{ineq:est}, we
obtain
$$\ip{\log_{p}y_0}{\bra{\xia}}
= \frac{1}{\mu(B)} \int_{B} \arrowprod{py_0}{px} d\mu(x) \to \arrowprod{py_0}{px_0}$$
as $\delta \to 0+$. This proves equation~\eqref{eq:cangle=angle}.
\end{proof}

\begin{lem}
Let $(X, d)$ be an Alexandrov spaces with curvature $\ge\k$. Suppose
that we have a point $p$ of $X$ and a subset $Y \subset X$ such that:
\begin{itemize}
\item
For any $x \in Y \setminus\{p\}$, there exist unique $\arrow{x}{p}$ and $-\arrow{x}{p}$
in $\Sigma_p$.
\item
For any $x, y \in Y \setminus\{p\}$, $\cangle(p; x, y) = \angle(\arrow{x}{p}, \arrow{y}{p})$.
\item
There is an isometric embedding map $F: [p, Y] \to M_\k$.
\end{itemize}
Then $F$ is extended to the isometric embedding of $\conv [p, Y]$ into $M_\k$.
\end{lem}

\begin{proof}
Take $x_i \in Y$ for $i = 1, 2, 3$ with $\arrow{x_1}{p} \neq \pm\arrow{x_2}{p}$
and fix a small positive number $\delta > 0$.
Let $\gamma_i : [0, l_i] \to X$ be the geodesic from $p$ to $x_i$
and $x_{i-\delta} := \gamma_i(l_i - \delta)$ for
$i = 1, 2$. Take any sequence $\bra{x^k} \subset J_p$ such that $\bra{x^k} \in (x_{1-\delta}, x_{2-\delta})$.

Let $\eta^k$ be the geodesic from $p$ to $x^k$. Then by Lemma~\ref{lem:rigid},
$$\cangle(p; \gamma_i(\epsilon), \bra{\eta^k(\epsilon)})
= \cangle(p; x_i, \bra{x^k})$$
for any $\epsilon > 0$ and $i = 1, 2$. The triangle comparison yields that
\begin{equation}
\angle\(\pm \arrow{x_i}{p}, \barrow{x^k}{p}\)
\ge \angle(\pm \arrow{\tx_i}{\tp}, \arrow{\tx}{\tp})
\end{equation}
for $i = 1, 2, 3$. Then this inequality must be equality. By the same
argument as in the proof of Theorem~\ref{thm:B}, we know that $\barrow{x^k}{p}$ 
and $\bra{x^k}$
are
Cauchy sequences in $\Sigma_p$ and in $X$, respectively. The sequence
$\bra{x^k}$
converges
to the unique point $x_{12-\delta} \in (x_{1-\delta}, x_{2-\delta})$. Letting $\delta \to 0+$, we can find the point
$x_{12} \in (x_1, x_2)$.

Finally, we verify that $\cangle(p; x_{12}, x_3) = \cangle(\tp; \tx_{12}, \tx_3)$. The angle monotonicity
implies that
$$\cangle(p; x_{12}, x_3)
\le \angle(\arrow{x_{12}}{p}, \arrow{x_3}{p})
= \cangle(\tp; \tx_{12}, \tx_3),$$
while the reverse inequality follows from the triangle comparison.

This leads to the isometry between $\conv [p, Y]$ and $\conv F([p, Y])$ as in the
proof of Theorem~\ref{thm:B}. This concludes the proof of the lemma.
\end{proof}
The above two lemmas complete the proof of Theorem~\ref{thm:A}.
\end{proof}

\section{Applications}\label{sect:9}
In this final section, we give a couple of applications of our main theorems.
At first, we consider the following inequality which can be found in Villani's
book~\cite[(8.45)]{Vi}. The proof is done by using the triangle comparison inequality
three times.
\begin{prop}
Let $(X, d)$ be an Alexandrov space of non-negative curvature.
Suppose that $\gamma, \eta : [0, 1] \to X$ be two constant speed geodesics. Then for any
$t \in (0, 1)$,
\begin{align}\label{ineq:Fig}
\begin{split}
d(\gamma(t), \eta(t))^2
\ge& (1 - t)^2 d(\gamma_0, \eta_0)^2 + t^2 d(\gamma_1, \eta_1)^2\\
&+ t(1 - t)
\Bigl[
d(\gamma_1, \eta_0)^2 + d(\gamma_0, \eta_1)^2 - d(\gamma_0, \gamma_1)^2 - d(\eta_0, \eta_1)^2 \Bigr],
\end{split}
\end{align}
where we abbreviated $\gamma(i), \eta(i)$ to $\gamma_i, \eta_i$ for $i = 0, 1$.
\end{prop}

Now it is easy to prove the following.
\begin{prop}
Suppose that we have the equality in~\eqref{ineq:Fig} for some $t \in (0, 1)$
and in addition that there is a midpoint $p \in (\gamma(t), \eta(t))$. Then $\conv [p, \gamma \cup \eta]$ is
isometric to a closed convex set in the Euclidean space $\R^3$.
\end{prop}

\begin{proof}
We are going to see that the proposition follows from Theorem~\ref{thm:B}. Since
we have the equality in~\eqref{ineq:Fig}, we know that the quadruples $(\gamma(t); \eta_0, \eta(s), \eta_1)$
and $(\eta(t); \gamma_0, \gamma(s), \gamma_1)$ are isometric to the ones in the Euclidean plane for any
$s \in (0, 1)$.

We let $\lambda(\gamma_0) = \lambda(\eta_0) := 1 - t $ and $\lambda(\gamma_1) = \lambda(\eta_1) := t$.
Then, with $Y := \bra{\gamma_i, \eta_i \,|\, i = 0, 1}$, we have
$$\sum_{x, y\in Y} \lambda(x)\lambda(y) \arrowprodnone{px}{py}_0
= d(p, \gamma(t))^2 + d(p, \eta(t))^2 + 
2\arrowprodnone{p\gamma(t)}{p\eta(t)}_0 = 0.$$
Now applying Theorem~\ref{thm:B} proves the proposition.
\end{proof}

Next we consider the packing radius of positively curved Alexandrov spaces.
\begin{defi}
Let $(X, d)$ be a metric space and $q \ge 2$. We define its \textit{$q$-th packing
radius} $\pack(X)$ by
\begin{equation}\label{eq:pack}
\pack(X) := \frac{1}{2} \sup \bra{\min_{1\le i<j\le q} d(x_i, x_j) \Bigm| (x_i) \in X^q }.
\end{equation}
The sequence $(x_i) \in X^q$ is called a \textit{$q$-th packer} when it attains the supremum
in~\eqref{eq:pack}.
\end{defi}

\begin{prop}
Let $(X, d)$ be an Alexandrov space with curvature $\ge 1$. Then
\begin{enumerate}
\item
the $q$-th packing radius $\pack(X)$ of $X$ does not exceed $\frac{1}{2}\arccos \frac{1}{1-q}$;
that of the round sphere $\S^n$ of constant curvature $1$ and dimension $n \ge q-2$.
\item
If $\pack(X) = \pack(\S^{q-2})$ and there exists a $q$-th packer, then $X$ is isometric
to the spherical join $\S^{q-2} * Y$ for some Alexandrov space $Y$ with
curvature $\ge 1$.
\end{enumerate}
\end{prop}

This proposition was established by Grove--Wilhelm~\cite{GW} for finite dimensional
Alexandrov spaces. Now that we have Theorem~\ref{thm:B}, we can prove this for possibly
infinite dimensional Alexandrov spaces. Part (1) follows from Lang--Schroeder--
Sturm's inequality~\eqref{ineq:LSS}. If we have a $q$-th packer $(x_i) \in X^q$ giving $\pack(X) =
\pack(\S^{q-2})$, inequality~\eqref{ineq:LSS} implies that $d(x_i, x_j) = \arccos \frac{1}{1-q}$ for any $i \ne j$,
and Theorem~\ref{thm:B} yields that $x_i$'s are contained in a subset isometric to the round
sphere $\S^{q-2}$. Then we only have to appeal to the maximum diameter theorem
(e.g.~\cite{Mi}).

\begin{rem}
In a Hilbert space $\H$, every probability measure $\mu \in P_1(\H)$ admits
its barycenter (or center of mass). If $X = \H$ and $\k = 0$, we have the equality
in the Lang--Schroeder--Sturm inequality~\eqref{ineq:LSS} if and only if the point p is the
barycenter of $\mu \in P_1(\H)$.

Recently, Ohta~\cite{Oh} investigated the properties of barycenters of probability
measures on Alexandrov spaces with lower curvature bound. The relation between
the barycenter of $\mu$ and the point $p$ which gives the equality in~\eqref{ineq:LSS} is yet
not clear to the author for general Alexandrov spaces.
\end{rem}

We close this paper by giving the following theorem. Since the proof is a simple
modification of that of Theorem~\ref{thm:A}, we leave it to the interested reader.

\begin{thm}
Let $(C_p, |\cdot|)$ be the tangent cone at a point p of an Alexandrov
space with lower curvature bound. Then
\begin{enumerate}
\item
for any Borel probability measure $\mu$ in $P_1(C_p)$, we have
\begin{equation}\label{ineq:intLS}
\int_{C_p}\int_{C_p} \ip{\xi}{\eta} d\mu(\xi)d\mu(\eta) \ge 0.
\end{equation}
\item
If we have the equality in~\eqref{ineq:intLS}, then the linear hull of the support of $\mu$ in
$C_p$ is isometric to a Hilbert space.
\end{enumerate}
\end{thm}

\subsection*{Acknowledgments}
The author would like to express his gratitude to Takao Yamaguchi
and Koichi Nagano for their comments and discussions. Special thanks
are due to Ayato Mitsuishi.

\end{document}